\documentclass{amsart}
\usepackage{enumerate,amssymb,amsmath,graphics,epsfig}
\usepackage[colorlinks]{hyperref}
\usepackage{color}
\usepackage{subfigure}
\usepackage{xspace}
\usepackage{arydshln}

\let\oldv\v
\newcommand\RE{\mathbb{R}}
\newcommand\NA{\mathbb {N}}
\renewcommand\div{\mathop{\rm{div}}\nolimits}
\newcommand\Grad{\mathop{\boldsymbol\nabla}\nolimits}
\newcommand\Grads{\mathop{\boldsymbol\nabla_{\rm sym}}\nolimits}

\newcommand\Huo{H^1_0(\Omega)^d}
\newcommand\Hub{H^1(\B)^d}
\newcommand\Hubd{(H^1(\B)^d)'}
\newcommand\HhalfB{H^{1/2}(\B)^d}
\newcommand\Ldo{L^2_0(\Omega)}
\newcommand\Of{\Omega^f}
\newcommand\Os{\Omega^s}
\newcommand\dx{d\x}
\newcommand\ds{d\s}
\newcommand\dt{\Delta t}
\newcommand\tria{\mathcal{T}_h}
\newcommand\triaB{\mathcal{T}_h^{\B}}
\newcommand\Xb{\overline\X}
\newcommand\Gb{\overline\Gamma}
\newcommand\Vo{V_0}
\newcommand\Voh{V_{0,h}}
\newcommand\Vh{V_h}
\newcommand\Qh{Q_h}
\newcommand\Sh{S_h}
\newcommand\Lh{\LL_h}
\newcommand\lh{\llambda_h}
\newcommand\uh{\u_h}
\newcommand\ph{p_h}
\newcommand\Xh{\X_h}

\renewcommand\a{\mathbf{a}}
\renewcommand\c{\mathbf{c}}
\newcommand\f{\mathbf{f}}
\newcommand\g{\mathbf{g}}
\renewcommand\d{\mathbf{d}}
\newcommand\n{\mathbf{n}}
\newcommand\s{\mathbf{s}}
\renewcommand\u{\mathbf{u}}
\renewcommand\v{\mathbf{v}}
\newcommand\w{\mathbf{w}}
\newcommand\x{\mathbf{x}}
\newcommand\z{\mathbf{z}}
\newcommand\X{\mathbf{X}}
\newcommand\Y{\mathbf{Y}}
\newcommand\Z{\mathbf{Z}}
\newcommand\U{\mathbf{U}}
\newcommand\V{\mathbf{V}}
\newcommand\B{\mathcal B}

\newcommand\ssigma{\boldsymbol{\sigma}}
\newcommand\llambda{\boldsymbol{\lambda}}
\newcommand\LL{\boldsymbol{\Lambda}}
\newcommand\mmu{{\boldsymbol{\mu}}}
\newcommand{\F}{\mathbb{F}}
\renewcommand{\P}{\mathbb{P}}
\newcommand\VV{\mathbb{V}}
\renewcommand\AA{\mathbb{A}}
\newcommand\BB{\mathbb{B}}
\newcommand\K{\mathbb{K}}
\newcommand\KBB{\K_{\BB}}
\newcommand\KBBh{\K_{\BB,h}}
\newcommand\ucX{\u(\X(\cdot,t),t)}
\newcommand\uhcX{\uh(\X(\cdot,t),t)}
\newcommand\vcX{\v(\X(\cdot,t))}
\newcommand\ucXb{\u(\Xb)}
\newcommand\uhcXb{\uh(\Xb)}
\newcommand\vcXb{\v(\Xb)}
\newcommand\vhcXb{\v_h(\Xb)}
\newcommand\dr{\delta_\rho}
\newcommand{\Af}{A_f}
\newcommand{\As}{A_s}
\newcommand{\Bf}{B_f}
\newcommand{\Bft}{B_f^\top}
\newcommand{\Cs}{C_s}
\newcommand{\Cst}{C_s^\top}
\newcommand{\Cf}{C_f}
\newcommand{\Cft}{C_f^\top}
\newcommand{\zero}{0}
\newcommand{\CC}{\mathbb{C}}
\newcommand{\feibm}{\textsc{fe-ibm}\xspace}
\newcommand{\dlmibm}{\textsc{dlm-ibm}\xspace}
\newcommand{\cfl}{\textsc{cfl}\xspace}
\newcommand\corrigendum[1]{{\color{black}#1}}
\definecolor{BLUE}{rgb}{0,0,1}
\newcommand\corrig[1]{{\color{black}#1}}
\definecolor{cardinal}{rgb}{0.84,0.04,0.33}
\theoremstyle{plain}
\newtheorem{thm}{Theorem}
\newtheorem{proposition}[thm]{Proposition}
\newtheorem{lemma}[thm]{Lemma}

\newtheorem{problem}{Problem}

\theoremstyle{remark}

\newtheorem*{remark}{Remark}

\begin{document}
\title[DLM/FD for FSI]
{A fictitious domain approach with Lagrange multiplier 
for fluid-structure interactions}
\author{Daniele Boffi}
\address{Dipartimento di Matematica ``F. Casorati'', Universit\`a di Pavia,
Italy}
\email{daniele.boffi@unipv.it}
\urladdr{http://www-dimat.unipv.it/boffi/}
\author{Lucia Gastaldi}
\address{DICATAM, Universit\`a di Brescia, Italy}
\email{lucia.gastaldi@unibs.it}
\urladdr{http://lucia-gastaldi.unibs.it}
\begin{abstract}
We study a recently introduced formulation for fluid-structure interaction
problems which makes use of a distributed Lagrange multiplier in the spirit of
the fictitious domain method. The time discretization of the problem leads to
a mixed problem for which a rigorous stability analysis is provided. The
finite element space discretization is discussed and optimal convergence
estimates are proved.
\end{abstract}
\maketitle
\section{Introduction}
\label{se:intro}
Numerical schemes for fluid-structure interaction problems include interface
fitted meshes (thus requiring suitable remeshing in order to keep the fluid
computational grid aligned with the interface) or interface non-fitted meshes
(allowing to keep the fluid computational grid fixed and independent from the
position of the solid).

The immersed boundary method (see~\cite{PeAN} for a review) is a typical
example of non-fitted schemes. It has been introduced in the 70's for the
simulation of biological problems related to the blood flow in the heart and
it has been extended to finite elements in a series of papers starting
from~\cite{feibm1} by using a variational approach (\feibm) and~\cite{WangLiu}
where finite elements and reproducing kernel particle methods are combined.
The \feibm allows for \emph{thick} (i.e., of codimension zero) or
\emph{thin} (i.e., of codimension one) structures. In particular, in~\cite{BGHP}
the original fiber-like description of the structure has been abandoned in
favor of a more natural and intrinsically \emph{thick} modeling of the solid
domain. With this representation, a unified treatment of immersed structures
is possible in any combination of dimensions. Two novelties have been
introduced in~\cite{costanzo}: a compressible model for the structure has been
considered, and the motion of the solid has been taken care with a fully
variational approach.

In~\cite{bcg2015} a new formulation (\dlmibm) for fluid-structure interaction
problems has been introduced based on the \feibm which makes use of a
distributed Lagrange multiplier in the spirit of the fictitious domain method
(see, for instance,~\cite{girglo1995,girglopan,glokuz2007}). The \dlmibm in
the codimension one case has some similarities with the so called
``immersogeometric'' method recently introduced in~\cite{kamensky,kamensky2}.
An important feature of the \dlmibm, as it has been shown in~\cite{bcg2015},
is that its semi-implicit time discretization results to be unconditionally
stable as opposed to the standard \feibm where a suitable \cfl condition has
to be satisfied (see~\cite{BGHM3AS,heltai,BCG2011}).

The time discretization of the problem leads to a saddle point problem (see
Problem~\ref{pb:saddlepoint} and its discrete counterpart
Problem~\ref{pb:saddlepointh}).
The main contribution of this paper is the rigorous analysis of
Problems~\ref{pb:saddlepoint} and~\ref{pb:saddlepointh}: it is shown that the
saddle point problems are stable and that the discrete solution converges
optimally towards the continuous one. Suitable conditions on the solid mesh
are stated: in the case of codimension zero structures, the mesh is assumed to
allow $H^1$ stability for the $L^2$ projection (more detailed description of
this assumption is given in the discussion after
Proposition~\ref{pr:infsuph}); in the case of codimension one structures, the
solid meshsize is assumed to satisfy a suitable compatibility condition with
respect to the fluid one.

The structure of the paper is the following: in Section~\ref{se:pb} we
introduce the problem and derive the \dlmibm in the case of solids of
codimension zero.
This is a new approach and provides an interesting result, since the
\dlmibm , previously derived as a modification of the \feibm, is now seen as a
natural fictitious domain formulation originating from a strong form of a
fluid-structure interaction problem. Section~\ref{se:dt} recalls the
semi-implicit time discretization of the \dlmibm and the known energy
estimates. Section~\ref{se:mixed} is the core of our paper in the case of
thick structures, presenting the analysis of the mixed problem and of its
numerical approximation. Finally, Section~\ref{se:thin} performs the same
analysis of the mixed problem in the case where thin structures are
considered.

\section{Fictitious domain approach in the case of a thick solid immersed in a
fluid}
\label{se:pb}
The fluid-structure interaction system that we are going to analyze in this
paper consists of a solid elastic body immersed in a fluid.
We refer to a \emph{thick} solid when it occupies a domain of
codimension zero, and to a \emph{thin} solid when the corresponding domain can
be reduced to a region of codimension one in the fluid by using standard
assumptions on the behavior of the involved physical quantities.
This case will be treated in Section~\ref{se:thin}.

Let $\Of_t\subset\RE^d$ and $\Os_t\subset\RE^d$ with $d=2,3$ be the time
dependent regions occupied by the fluid and the structure, respectively.
We set $\Omega$ the interior of
$\overline{\Omega}^f_t\cup\overline{\Omega}^s_t$ and
assume that $\Omega$ is a fixed domain.
We denote by $\Gamma_t=\partial\Of_t\cap\partial\Os_t$ the moving interface 
between the fluid and the solid regions. For simplicity, we assume that
the structure is immersed in the fluid so that
$\partial\Os_t\cap\partial\Omega=\emptyset$.

Assuming that both the fluid and the solid material are incompressible,
the fluid-structure interaction problem can be written in a very general form 
as follows:
\begin{equation}
\label{eq:general}
\aligned
&\rho_f\dot\u_f=\div\ssigma_f  &&\text{in }\Of_t\\
&\div\u_f=0 &&\text{in }\Of_t\\
&\rho_s\dot\u_s=\div\ssigma_s &&\text{in }\Os_t\\
&\div\u_s=0 &&\text{in }\Os_t\\
&\u_f=\u_s &&\text{on }\Gamma_t\\
& \ssigma_f\n_f=-\ssigma_s\n_s &&\text{on }\Gamma_t.
\endaligned
\end{equation}
The system can be complemented with the following initial and boundary
conditions on $\partial\Omega$:
\begin{equation}
\label{eq:ibc}
\aligned
&\u_f(0)=\u_{f0} && \text{on }\Of_0,\\
&\u_s(0)=\u_{s0} && \text{on }\Os_0,\\
&\u_f(t)=0 && \text{on }\partial\Omega.
\endaligned
\end{equation}
In~\eqref{eq:general} $\u$, $\ssigma$, and $\rho$ denote velocity, stress
tensor, and mass density, respectively.
The subscript $f$ or $s$ refers to fluid or solid.
We assume moreover that $\rho_f$ and $\rho_s$ are positive constants.

In the following we introduce the constitutive laws for fluid and solid
materials and derive the variational formulation
of~\eqref{eq:general}-\eqref{eq:ibc}.

First of all, let us define some functional spaces we shall work with.
For a domain $\omega$ we denote by $L^2(\omega)$ the space of 
square integrable functions in $\omega$, endowed with the norm 
$\|v\|_{0,\omega}^2=\displaystyle\int_\omega |v|^2dx$ and the corresponding
scalar product denoted by $(\cdot,\cdot)_{\omega}$. Then $H^1(\omega)$ is
the space of functions belonging to $L^2(\omega)$ together with their gradient;
then $\|v\|_{1,\omega}^2=\|v\|_{0,\omega}^2+\|\nabla v\|_{0,\omega}^2$
defines the norm in $H^1(\omega)$.
We denote by $H^1_0(\omega)$ the subspace of $H^1(\omega)$ of functions
vanishing on the boundary of $\omega$ and by $L^2_0(\omega)$ the subspace of
$L^2(\omega)$ of functions with zero mean value.
When $\omega=\Omega$ we omit the subscripts $\Omega$.

The equations in~\eqref{eq:general} are written using the Eulerian description,
but the deformation of the solid is usually described in the Lagrangian
framework. For this, we consider $\Os_t$ as the image of a reference domain
$\B\subset\RE^d$. For every $t\in[0,T]$, we associate points 
$\s\in\B$ and $\x\in\Os_t$ via a family of mappings $\X(t):\B\to\Os_t$.
We refer to $\s\in\B$ as the material or Lagrangian coordinate and to
$\x=\X(\s,t)$ as the spatial or Eulerian coordinate with $\x\in\Os_t$.
We assume that $\X$ fulfills the following conditions:
$\X(t)\in W^{1,\infty}(\B)$, $\X(t)$ is one to one, and
there exists a constant $\gamma$ such that for all $t\in[0,T]$ 
$\|\X(\s_1,t)-\X(\s_2,t)\|\ge\gamma\|\s_1-\s_2\|$ for all $\s_1,\s_2\in\B$.
Note that this requirements imply that $\X(t)$ is invertible with Lipschitz
inverse. This in particular implies that $\Y\in\Hub$ if and only if
$\v=\Y(\X^{-1}(t))\in H^1(\Os_t)^d$.
The deformation gradient is defined as $\F=\Grad_s\X$ and we indicate with
$|\F|$ its determinant. 
In~\eqref{eq:general}, the dot over the velocity denotes the material
time derivative. In the fluid, using
the Eulerian description, we have 
$\dot\u_f=\partial\u_f/\partial t+\u_f\cdot\Grad\u_f$.
In the solid, the Lagrangian framework is preferred and the spatial description
of the material velocity reads
\begin{equation}
\label{eq:materialvelocity}
\u_s(\x,t)=\frac{\partial\X(\s,t)}{\partial t}\Big|_{\x=\X(\s,t)}
\end{equation}
so that $\dot\u_s(\x,t)=\partial^2\X(\s,t)/\partial t^2|_{\x=\X(\s,t)}$.
Thanks to the incompressibility condition for fluid and solid, expressed by the
divergence free condition in~\eqref{eq:general}, it results that $|\F|$ is
constant in time and equals its initial value. In particular, if the reference
domain $\B$ coincides with the initial position of the solid $\Os_0$ one has
that $|\F|=1$ for all $t$.

Let us introduce now the constitutive laws for fluid and solid materials, in
order to model the stress tensor.
We consider a Newtonian fluid characterized by the usual Navier--Stokes stress
tensor
\begin{equation}
\label{eq:sigmaf}
\ssigma_f=-p_f\mathbb{I}+\nu_f\Grads\u_f,
\end{equation}
where $\Grads\u=(1/2)\left(\Grad\u_f+(\Grad\u_f)^\top\right)$ is the symmetric
gradient and $\nu_f$ represents the viscosity of the fluid.
The solid material is assumed to be viscous-hyperelastic, so that the Cauchy
stress tensor can be represented as the sum $\ssigma_s=\ssigma_s^f+\ssigma_s^s$
of a fluid-like part, with viscosity $\nu_s$,
\begin{equation}
\label{eq:sigmas}
\ssigma_s^f=-p_s\mathbb{I}+\nu_s\Grads\u_s
\end{equation}
and an elastic part $\ssigma_s^s$. 
By changing variable from Eulerian to Lagrangian, we express $\ssigma_s^s$
in term of the first Piola--Kirchhoff stress tensor $\P$: 
\begin{equation}
\label{eq:Piola}
\P(\F(\s,t))=
|\F(\s,t)|\ssigma_s^s(\x,t)\F^{-\top}(\s,t)\qquad\text{for }\x=\X(\s,t).
\end{equation}
On the other hand, hyperelastic materials are characterized by a positive energy
density $W(\F)$ which is related to the Piola--Kirchhoff stress tensor as follows:
\begin{equation}
\label{eq:1st-pk}
(\P(\F(\s,t))_{\alpha i}=
\frac{\partial W}{\partial\F_{\alpha i}}(\F(\s,t))
=\left(\frac{\partial W }{\partial \F}(\F(\s,t))\right)_{\alpha i},
\end{equation}
where $i = 1,\ldots,m$ and $\alpha=1,\ldots,d$. The elastic potential energy of
the body is given by:
\begin{equation}
\label{eq:potenergy}
E\left(\X(t)\right)=\int_\B W(\F(s,t))\ds.
\end{equation}
Let $\v\in\Huo$ be given. We multiply the first equation
in~\eqref{eq:general} by $\v|_{\Of_t}$, integrate over $\Of_t$, and integrate
by parts; analogously, we multiply the third equation by
$\v|_{\Os_t}$ and integrate over $\Os_t$, and integrate by parts.
Summing up the two equations and taking into account the transmission
conditions on $\Gamma_t$,  we obtain the following equation which corresponds
to the principle of virtual work:
\[
\int_{\Of_t}\rho_f\dot\u_f\ \v\dx+\int_{\Os_t}\rho_s\dot\u_s\ \v\dx
+\int_{\Of_t}\ssigma_f:\Grads\v\dx+\int_{\Os_t}\ssigma_s:\Grads\v\dx=0.
\]
Introducing the models~\eqref{eq:sigmaf}-\eqref{eq:sigmas}
and taking into account~\eqref{eq:materialvelocity}
and~\eqref{eq:Piola}, we arrive to the following equation
\begin{equation}
\label{eq:virtualwork2}
\aligned
\int_{\Of_t}&\rho_f\dot\u_f\ \v\dx+
\int_{\B}\rho_s\frac{\partial^2\X}{\partial t^2}\ \v(\X(\s,t))\ds
+\int_{\Of_t}\nu_f\Grads\u_f:\Grads\v\dx\\
&-\int_{\Of_t}p_f\div\v\dx
+\int_{\Os_t}\nu_s\Grads\u_s:\Grads\v\dx\\
&+\int_\B\P(\F(\s,t)):\Grad_s\v(\X(\s,t))\ds
-\int_{\Os_t}p_s\div\v\dx
=0 \quad\forall\v\in\Huo,
\endaligned
\end{equation}
where we used the standard notation
$\mathbb{D}:\mathbb{E}=
\sum_{\alpha,i=1}^d\mathbb{D}_{\alpha i}\mathbb{E}_{\alpha i}$ for all tensors
$\mathbb{D}$ and $\mathbb{E}$.

We observe that in~\eqref{eq:virtualwork2} $p_f$ and $p_s$ are not uniquely
determined. In fact, if we take $p_f+c_f$ and $p_s+c_s$ instead of $p_f$ and
$p_s$, respectively, the left hand side of~\eqref{eq:virtualwork2} does not
change if $c_f=c_s$. To avoid this situation we impose that
\begin{equation}
\label{eq:p_unique}
\int_{\Of_t} p_f\dx+\int_{\Os_t} p_s\dx=0.
\end{equation} 
At the end, the incompressibility condition for both materials can be written
in variational form as:
\begin{equation}
\label{eq:divfree}
\int_{\Of_t}\div\u_f q\dx+\int_{\Os_t}\div\u_s q\dx=0\quad\forall q\in\Ldo.
\end{equation}

Then the fluid-structure interaction problem can be written in the following
form.
\begin{problem}
\label{pb:orig}
For $t\in]0,T]$ find $\u_f(t)\in H^1(\Of_t)^d$, $p_f(t)\in L^2(\Of_t)$,
$\u_s(t)\in H^1(\Os_t)^d$, $p_s(t)\in L^2(\Os_t)$, and
$\X(t)\in\Hub$ such that $\u_f(t)=\u_s(t)$ on $\Gamma_t$, and
equations~\eqref{eq:virtualwork2}, \eqref{eq:p_unique}, \eqref{eq:divfree}, 
\eqref{eq:materialvelocity}, and~\eqref{eq:ibc} are satisfied together with
$\X(0)=\X_0$ on $\B$, where $\X_0:\B\to\Os_0$.
\end{problem}
\begin{remark}
Thanks to~\eqref{eq:materialvelocity}, the initial condition for $\u_s$
provides also an initial condition for $\partial\X/\partial t$.
\end{remark}

In the following, we use a \emph{fictitious domain} approach with a distributed
Lagrange multiplier in order to
rewrite the variational formulation of the problem.
Namely, we extend the fluid velocity and pressure into the solid domain
by introducing new unknowns with the following meaning:
\begin{equation}
\label{eq:FDunknowns}
\u=\left\{
\begin{array}{ll}
\u_f&\text{ in } \Of_t\\
\u_s&\text{ in } \Os_t
\end{array}
\right.,\quad
p=\left\{
\begin{array}{ll}
p_f&\text{ in } \Of_t\\
p_s&\text{ in } \Os_t
\end{array}
\right.
\end{equation}
with the condition that the material velocity of the solid is equal to the
velocity of the fictitious fluid, that is 
\begin{equation}
\label{eq:constraint}
\frac{\partial\X(\s,t)}{\partial t}=\u(\X(\s,t),t) \quad\text{for }\s\in\B.
\end{equation}
This equation, which governs the evolution of the immersed solid, 
represents a constraint for the problem, therefore
we enforce it in variational form by introducing a Lagrange multiplier as
follows. 
Let $\LL$ be a functional space to be defined later on and 
$\c:\LL\times\Hub\to\RE$ a bilinear form such that 
\begin{equation}
\label{eq:defc}
\aligned
&\c \text{ is continuous on } \LL\times\Hub\\
&\c(\mmu,\Z)=0 \text{ for all } \mmu\in\LL \text{ implies } \Z=0.
\endaligned
\end{equation}
For example we can take as $\LL$ the dual space of $\Hub$ and define $\c$ as
the duality pairing between $\Hub$ and $\Hubd$, that is:
\begin{equation}
\label{eq:defc1}
\c(\mmu,\Y)=\langle\mmu,\Y\rangle\quad\forall\mmu\in\Hubd,\ \Y\in\Hub.
\end{equation}
Alternatively, one can set $\LL=\Hub$ and define
\begin{equation}
\label{eq:defc2}
\c(\mmu,\Y)=\left(\Grad_s\mmu,\Grad_s\Y\right)_\B+\left(\mmu,\Y\right)_\B
\quad\forall\mmu,\ \Y\in\Hub.
\end{equation}
Relation~\eqref{eq:constraint} can now be written in variational form as:
\begin{equation}
\label{eq:varconstr}
\c\left(\mmu,\ucX-\frac{\partial\X}{\partial t}(t)\right)
   =0 \quad\forall\mmu\in \LL.
\end{equation}
Then the problem can be formulated in the following weak form.
\begin{problem}
\label{pb:pbvar}
Given $\u_0\in\Huo$ and $\X_0\in W^{1,\infty}(\B)^d$,
for almost every $t\in]0,T]$ find
$(\u(t),p(t))\in\Huo\times\Ldo$, $\X(t)\in\Hub$,
and $\llambda(t)\in\LL$ such that it holds
\begin{subequations}
\begin{alignat}{2}
  &\rho_f\frac d {dt}(\u(t),\v)+b(\u(t),\u(t),\v)+a(\u(t),\v)\qquad\notag&&\\
  &\qquad-(\div\v,p(t))+\c(\llambda(t),\vcX)=0
   &&\quad\forall\v\in\Huo
     \label{eq:NS1_DLM}\\
  &(\div\u(t),q)=0&&\quad\forall q\in\Ldo
     \label{eq:NS2_DLM}\\
  &\dr\left(\frac{\partial^2\X}{\partial t^2}(t),\Y\right)_{\B}
+(\P(\F(t)),\Grad_s\Y)_{\B}-\c(\llambda(t),\Y)=0&&\quad\forall\Y\in\Hub
     \label{eq:solid_DLM}\\
  &\c\left(\mmu,\ucX-\frac{\partial\X}{\partial t}(t)\right)
   =0 &&\quad\forall\mmu\in \LL
     \label{eq:ode_DLM}\\
  &\u(0)=\u_0\quad\mbox{\rm in }\Omega,\qquad\X(0)=\X_0\quad\mbox{\rm in }\B.
     \label{eq:ci_DLM}
\end{alignat}
\end{subequations}
\end{problem}
Here $\dr=\rho_s-\rho_f$ and
\[
\aligned
&a(\u,\v)=(\nu\Grads\u,\Grads\v)\quad \text{with }
\nu=\left\{\begin{array}{ll}
\nu_f&\text{in }\Of_t\\
\nu_s&\text{in }\Os_t
\end{array},\right.
\\
&b(\u,\v,\w)=
\frac{\rho_f}2\left((\u\cdot\Grad\v,\w)-(\u\cdot\Grad\w,\v)\right).
\endaligned
\]
We assume that $\nu\in L^\infty(\Omega)$ and that there exists a positive
constant $\nu_0>0$ such that $\nu\ge\nu_0>0$ in $\Omega$.
\begin{remark}
In the literature of the Immersed Boundary Method, it is generally assumed
that the fluid and the solid visco-hyperelastic materials have the same
viscosity.
If this is not the case, the integral in the definition of $a$ has to be 
decomposed into the integral over $\Of_t$ and $\Os_t$. Therefore, 
in the finite element discretization, the associated stiffness matrix has to be
recomputed at each time step as $\nu$ is discontinuous along the moving
interface $\Gamma_t$.
This could be avoided if we treat this term using the fictitious domain method  
as it is done for the first integral
containing the time derivative of $\u$.
Then Problem~\ref{pb:pbvar} takes the following form.
\begin{problem}
\label{pb:pbvar2}
Given $\u_0\in\Huo$ and $\X_0\in W^{1,\infty}(\B)^d$,
find $(\u(t),p(t))\in\Huo\times\Ldo$, $\X(t)\in\Hub$, and
$\llambda(t)\in\LL$,
such that for almost every $t\in]0,T]$ it holds
\[
\aligned
  &\rho_f\frac d {dt}(\u(t),\v)+b(\u(t),\u(t),\v)+\tilde a(\u(t),\v)
&&\\
  &\qquad-(\div\v,p(t))+\c(\llambda(t),\vcX)=0
   &&\quad\forall\v\in\Huo
\\
  &(\div\u(t),q)=0&&\quad\forall q\in\Ldo
\\
  &\dr\left(\frac{\partial^2\X}{\partial t^2}(t),\Y\right)_{\B}+
   d\left(\frac{\partial\X}{\partial t}(t),\Y\right)\\
&+(\P(\F(t)),\Grad_s\Y)_{\B}-\c(\llambda(t),\Y)=0&&\quad\forall\Y\in\Hub
\\
  &\c\left(\mmu,\ucX-\frac{\partial\X}{\partial t}(t)\right)
   =0 &&\quad\forall\mmu\in \LL
\\
  &\u(0)=\u_0\quad\mbox{\rm in }\Omega,\qquad\X(0)=\X_0\quad\mbox{\rm in }\B,
\endaligned
\]
\end{problem}
where $\tilde a(\u,\v)=\nu_f(\Grads\u,\Grads\v)$ and
\[
\aligned
d(\X&,\Y)=\\
&\frac12\int_\B(\nu_s-\nu_f)
\left(\Grad_s\X\F^{-1}+\F^{-\top}\Grad_s\X^\top\right):
\left(\Grad_s\Y\F^{-1}+\F^{-\top}\Grad_s\Y^\top\right)|\F|\ds.
\endaligned
\]
For the sake of simplicity, in the rest of this paper we are going to
consider a constant viscosity throughout the domain.
\end{remark}

First of all we show that Problems~\ref{pb:orig} and~\ref{pb:pbvar} 
are equivalent.
\begin{thm}
\label{th:equivalenza}
Let $(\u_f,\u_s,p_f,p_s,\X)$ be a solution of Problem~\ref{pb:orig}, 
such that $\X(t)\in W^{1,\infty}(\B)^d$ and $\X(t):\B\to\Os_t$ is one to one,
then setting $(\u,p)$ as in~\eqref{eq:FDunknowns}, there exists
$\llambda(t)\in\LL$ such that $(\u,p,\X,\llambda)$ is a solution
of Problem~\ref{pb:pbvar}. 

Conversely, let $(\u,p,\X,\llambda)$ be a solution of Problem~\ref{pb:pbvar}
such that $\X(t)\in W^{1,\infty}(\B)^d$ and $\X(t):\B\to\Os_t$ is one to one.
Set $\u_f(t)=\u(t)|_{\Of_t}$, $p_f(t)=p(t)|_{\Of_t}$,
$\u_s(t)=\u(t)|_{\Os_t}$, $p_s(t)=p(t)|_{\Os_t}$, then
$(\u_f,\u_s,p_f,p_s,\X)$ is a solution of Problem~\ref{pb:orig}.
\end{thm}
\begin{proof}
Let $(\u_f,\u_s,p_f,p_s,\X)$ be a solution of Problem~\ref{pb:orig}.
Taking into account that $\partial\Os_t\cap\partial\Omega=\emptyset$, the third
condition in~\eqref{eq:ibc} gives $\u=0$ on $\partial\Omega$. 
From~\eqref{eq:FDunknowns}, we obtain that $\u\in\Huo$, since $\u_f(t)\in
H^1(\Of_t)^d$, $\u_s(t)\in H^1(\Os_t)^d$, and 
$\u_f(t)=\u_s(t)$ on $\Gamma_t$, and that $p\in\Ldo$ thanks
to~\eqref{eq:p_unique}.
Next~\eqref{eq:divfree} implies~\eqref{eq:NS2_DLM}, while
\eqref{eq:defc}, and~\eqref{eq:varconstr} gives 
that~\eqref{eq:ode_DLM} holds true.
Setting 
\[
\u_0=\left\{
\begin{array}{ll}
\u_{f0} & \text{in }\Of_0\\
\u_{s0} & \text{in }\Os_0,
\end{array}
\right.
\]
the initial conditions~\eqref{eq:ci_DLM} are satisfied.
It remains to prove~\eqref{eq:NS1_DLM} and~\eqref{eq:solid_DLM}.
For this, we introduce $\llambda(t)\in\LL$ such that~\eqref{eq:NS1_DLM} is
satisfied.
Differentiating condition~\eqref{eq:constraint} with respect to time gives
$\dot\u(\x,t)=\partial^2\X(\s,t)/\partial t^2|_{\x=\X(\s,t)}$, hence recalling
the incompressibility of the structure, we have the following equality
\[
\int_{\Of_t}\rho_f\dot\u\ \v\dx=
\int_\B\rho_f\frac{\partial^2\X(\s,t)}{\partial t^2}\v(\X(\s,t))\ds.
\]
Then taking into account the definition of the forms $a$ and $b$
and~\eqref{eq:varconstr}, we have from~\eqref{eq:virtualwork2}:
\[
\int_{\B}\dr\frac{\partial^2\X}{\partial t^2}\ \v(\X(\s,t)\ds
+\int_\B\P(\F(\s,t)):\Grad_s\v(\X(\s,t))\ds
=\c(\llambda(t),\vcX) 
\]
for all $\v\in\Huo$. Since $\X(t):\B\to\Os_t$ is one to one 
and belongs to $W^{1,\infty}(\B)$, $\Y=\vcX$ is an arbitrary element of $\Hub$
and~\eqref{eq:solid_DLM} holds true.

Let us now prove the converse. Let $(\u,p,\X,\llambda)$ be a solution of
Problem~\ref{pb:pbvar}
and set $\u_f(t)=\u(t)|_{\Of_t}$, $p_f(t)=p(t)|_{\Of_t}$,
$\u_s(t)=\u(t)|_{\Os_t}$, $p_s(t)=p(t)|_{\Os_t}$.
From~\eqref{eq:ode_DLM} and~\eqref{eq:defc} we have that~\eqref{eq:constraint}
is fulfilled.
Using again the fact that $\X(t):\B\to\Os_t$ is one to one, we take
$\Y=\v(\X(t))$ in~\eqref{eq:solid_DLM} and sum it to~\eqref{eq:NS1_DLM}.
Equations~\eqref{eq:virtualwork2} and~\eqref{eq:divfree} follow from
the definition of $\u_f$, $\u_s$, $p_f$ and $p_s$. Moreover, we have that the
condition~\eqref{eq:p_unique} holds true since $p\in\Ldo$. 
It is easy to verify that the initial and boundary conditions are fulfilled.
\end{proof}
Thanks to the elastic properties of the viscous-hyperelastic material,
(see~\eqref{eq:1st-pk} and~\eqref{eq:potenergy}),
we have the following energy estimate (see~\cite{bcg2015} for the details).
\begin{proposition}
\label{pr:stab-cont}
Let us assume that $\dr\ge0$, that the potential energy density $W$ is a $C^1$
convex function
over the set of second order tensors and that for almost every $t\in[0,T]$,
the solution of Problem~\ref{pb:pbvar}
is such that $\X(t)\in(W^{1,\infty}(\B))^d$ 
with $\frac{\partial \X}{\partial t}(t)\in L^2(\B)^d$,
then the following equality holds true
\begin{equation}
\label{eq:energyest}
\frac{\rho_f}2\frac{d}{dt}||\u(t)||^2_0+\nu||\Grads\u(t)||^2_0+
\frac{\dr}2\frac{d}{dt}\left\|\frac{\partial \X(t)}{\partial t}\right\|^2_{0,\B}
+\frac{d}{dt}E(\X(t))=0.
\end{equation}
\end{proposition}
\begin{proof}
The proof is quite simple. Take $\v=\u(t)$, $q=p(t)$,
$\Y=\partial\X(t)/\partial t$, and $\mmu=\llambda(t)$ in
equations~\eqref{eq:NS1_DLM}-\eqref{eq:ode_DLM} respectively and sum. The
inequality~\eqref{eq:energyest} is achieved using~\eqref{eq:1st-pk}
and~\eqref{eq:potenergy} to estimate the second term arising
from~\eqref{eq:solid_DLM}.
\end{proof}

\section{Time semi-discretization}
\label{se:dt}
In this subsection we briefly recall the results of~\cite{bcg2015} related to
the time discretization of Problem~\ref{pb:pbvar}
and to the analysis of the stability of the resulting scheme. The presented
results are valid not only for thick structures but also for thin ones,
according to the formulation that will be presented in Section~\ref{se:thin}.

Given an integer $N>0$, set $\dt=T/N$ the time step and $t_n=n\dt$. 
For a given function $z$ depending on $t$ we denote by $z^n$ the approximation
of $z(t_n)$.

Problem~\ref{pb:pbvar} presents several nonlinear terms whose time
discretization by an implicit method would require the solution of a nonlinear
stationary system with non trivial computational cost. Therefore we adopt the
following semi-implicit time advancing scheme:
\begin{problem}
\label{pb:DLM_sd}
Given $\u_0\in\Huo$ and $\X_0\in W^{1,\infty}(\B)$, for $n=1,\dots,N$
find $(\u^n,p^n)\in\Huo\times\Ldo$, $\X^n\in\Hub$, and
$\llambda^n\in\LL$, such that
\begin{subequations}
\begin{alignat}{2}
  &\rho_f\left(\frac{\u^{n+1}-\u^n}{\dt},\v\right)
   +b(\u^n,\u^{n+1},\v)+a(\u^{n+1},\v)\notag&&\\
  &\qquad\qquad
-(\div\v,p^{n+1})+\c(\llambda^{n+1},\v(\X^n))=0
   &&\forall\v\in\Huo
   \label{eq:NS1sd}\\
  &(\div\u^{n+1},q)=0&&\forall q\in\Ldo
   \label{eq:NS2sd}\\
  &\dr\left(\frac{\X^{n+1}-2\X^n+\X^{n-1}}{\dt^2},\Y\right)_{\B}+
   (\P(\F^{n+1}),\Grad_s\Y)_{\B}\notag&&\\
&\qquad\qquad -\c(\llambda^{n+1},\Y)=0
     &&\forall\Y\in\Hub
   \label{eq:solidsd}\\
  &\c\left(\mmu,\u^{n+1}(\X^n)-\frac{\X^{n+1}-\X^n}{\dt}\right)
   =0 &&\forall\mmu\in\LL
\label{eq:odesd}\\
&\u^0=\u_0,\qquad\X^0=\X_0
\label{eq:cisd}
\end{alignat}
\end{subequations}
\end{problem}
In the second term of~\eqref{eq:solidsd} the implicit quantity $\P(\F^{n+1})$
might be difficult to compute, in such case different choices can be made. In
particular, if $\P$ is linear with respect to $\F$ this term does not cause
any trouble; otherwise it can be linearized.

In order to initialize equation~\eqref{eq:solidsd} we need to know the first
two values $\X^0$ and $\X^1$. These can be obtained from the initial conditions
taking into account~\eqref{eq:constraint} and~\eqref{eq:defc} as follows:
\[
\c\left(\mmu,\u_0(\X^0)-\frac{\X^1-\X^0}{\dt}\right)=0\quad\forall\mmu\in\LL.
\] 
Following the same lines of the proof of Proposition~\ref{pr:stab-cont}, we can
show the following unconditional stability for the time
advancing scheme, (see~\cite{bcg2015} for the details).  
\begin{proposition}
\label{pr:energy_sd}
Under the same assumptions as in Proposition~\ref{pr:stab-cont}, if
$\u^n\in\Huo$ and $\X^n\in\Hub$
for $n=0,\dots,N$ satisfy Problem~\ref{pb:DLM_sd} with
$\X^n\in(W^{1,\infty}(\B))^d$, then
the following estimate holds true for all $n=0,\dots,N-1$
\begin{equation}
\label{eq:energy_sd}
\aligned
&\frac{\rho_f}{2\dt}\left(\|\u^{n+1}\|^2_0-\|\u^n\|^2_0\right)+
\nu\|\Grads\u^{n+1}\|^2_0\\
&+\frac{\dr}{2\dt}\left(\left\|\frac{\X^{n+1}-\X^n}{\dt}\right\|^2_{0,\B}
-\left\|\frac{\X^n-\X^{n-1}}{\dt}\right\|^2_{0,\B}\right)+
\frac{E(\X^{n+1})-E(\X^n)}{\dt} \le0.
\endaligned
\end{equation}
\end{proposition}

\section{Analysis of the stationary problem}
\label{se:mixed}
We now focus on the stationary problem that we resolve at each time step and we
analyze its well-posedness and finite element discretization.
We consider a linear model for the Piola--Kirchhoff stress
tensor, that is 
\begin{equation}
\label{eq:PK-lin}
\P(\F)=\kappa\F=\kappa\Grad_s\X.
\end{equation}
In this case, we have that the energy density is $W(\F)=(\kappa/2)\F:\F$ and
the elastic potential energy is given by
\[
E(\X)=\frac{\kappa}2\int_\B \F:\F\ds=\frac\kappa2\int_\B |\Grad_s\X|^2\ds.
\]
With this simplification it is possible to apply the results on existence,
uniqueness, stability and error estimates of linear saddle point problems,
see~\cite{bbf}. We think that the results can be extended to the nonlinear case
with additional assumptions on the nonlinear terms. 
Hence we have the following saddle point problem.
\begin{problem}
\label{pb:saddlepoint}
Let $\Xb\in W^{1,\infty}(\B)^d$ be invertible with Lipschitz inverse
and $\overline\u\in L^\infty(\Omega)$.
Given $\f\in L^2(\Omega)^d$, $\g\in L^2(\B)^d$, and
$\mathbf{d}\in L^2(\B)^d$,
find $\u\in\Huo$, $p\in\Ldo$, $\X\in\Hub$, and
$\llambda\in{\color{blue}\LL}$ such that
\begin{equation}
\label{eq:stat}
\aligned
&\a_f(\u,\v)-(\div\v,p)+\c(\llambda,\vcXb)=(\f,\v)
&&\forall\v\in\Huo\\
&(\div\u,q)=0&&\forall q\in\Ldo\\
&\a_s(\X,\Y)-\c(\llambda,\Y)=(\g,\Y)_{\B}&&\forall\Y\in\Hub\\
&\c(\mmu,\ucXb-\X)= \c(\mmu,\mathbf{d})
&&\forall\mmu\in\LL
\endaligned
\end{equation}
where
\[
\aligned
&\a_f(\u,\v)=\alpha(\u,\v)+a(\u,\v)+b(\overline\u,\u,\v)&&\forall \u,\v\in\Huo\\
&\a_s(\X,\Y)=\beta(\X,\Y)_{\B}+\gamma(\Grad_s\X,\Grad_s\Y)_{\B}&&
\forall\X,\Y\in\Hub
\endaligned
\]
\end{problem}
It is easy to see that Problem~\ref{pb:saddlepoint}
corresponds to one step of Problem~\ref{pb:DLM_sd} if we take:
\[
\aligned
&\u=\u^{n+1}, \ p=p^{n+1},\  \X=\X^{n+1}/\dt,\  \llambda=\llambda^{n+1}\\
&\f=\frac{\rho_f}{\dt}\u^n\\
&\g=\frac{\dr}{\dt^2}\left(2\X^n-\X^{n-1}\right)\\
&\d=-\frac1{\dt}\X^n\\
&\alpha=\rho_f/\dt,\ \beta=\dr/\dt,\ \gamma=\kappa\dt
\endaligned
\]
and $\overline\X=\X^n$ and $\overline\u=\u^n$ in the nonlinear terms.

We remark that while $\alpha$ and $\gamma$ are strictly positive, the constant
$\beta$ might vanish when the densities in the solid and in the fluid are
equal.

For the sake of simplicity, in the sequel we neglect the convective term
associated with the trilinear form $b$.

\subsection{Well-posedness of Problem~\ref{pb:saddlepoint}}
\label{se:continuous}
\corrigendum{
Problem~\ref{pb:saddlepoint} fits in the framework of saddle point problems
and can be written in operator matrix form as follows:
\[
\left[
\begin{array}{ccc|c}
\Af&\Bft&\zero&\Cft\\
\Bf&\zero&\zero&\zero\\
\zero&\zero&\As&-\Cst\\
\hline
\Cf&\zero&-\Cs&\zero
\end{array}
\right]
\left[\begin{array}{c} \u\\ p \\ \X \\\hline \llambda \end{array}\right]
\left[\begin{array}{c} \f\\ 0 \\ \g \\ \hline \mathbf{d} \end{array}\right]
\]
with natural notation for the operators. In view of the analysis of this
problem and of its discrete counterpart, it is useful to rearrange the
variables, so that the matrix takes the form:
\[
\left[
\begin{array}{ccc|c}
\Af&\zero&\Cft&\Bft\\
\zero&\As&-\Cst&\zero\\
\Cf&-\Cs&\zero&\zero\\
\hline
\Bf&\zero&\zero&\zero
\end{array}
\right]
\left[\begin{array}{c} \u\\ \X \\ \llambda \\ \hline p \end{array}\right]
\left[\begin{array}{c} \f\\ \g \\ \mathbf{d}\\ \hline 0 \end{array}\right].
\]
This system corresponds to the following variational problem:
given $\f\in L^2(\Omega)$, $\g\in L^2(\B)$, and $\d\in\Hub$, find $\u\in\Huo$,
$\X\in\Hub$, $\llambda\in\LL$ and $p\in\Ldo$ such that
\begin{equation}
\label{eq:girato}
\aligned
&\a_f(\u,\v)+\c(\llambda,\vcXb)-(\div\v,p)=(\f,\v)
&&\forall\v\in\Huo\\
&\a_s(\X,\Y)-\c(\llambda,\Y)=(\g,\Y)_{\B}&&\forall\Y\in\Hub\\
&\c(\mmu,\ucXb-\X)= \c(\mmu,\mathbf{d})
&&\forall\mmu\in\LL\\
&(\div\u,q)=0&&\forall q\in\Ldo
\endaligned
\end{equation}
Let us introduce the Hilbert space
\begin{equation}
\label{eq:VV}
\VV=\Huo\times\Hub\times\LL
\end{equation}
endowed with the graph norm
\[
|||\V|||_{\VV}=\left(\|\v\|^2_1+\|\Y\|^2_{1,\B}
+\|\llambda\|^2_{\LL}\right)^{1/2},
\]
where $\V=(\v,\Y,\llambda)$ is a generic element of $\VV$.

We define the bilinear forms $\AA:\VV\times\VV\to\RE$ and
$\BB:\VV\times\Ldo\to\RE$
\begin{equation}
\label{eq:bilforms}
\aligned
&\AA(\U,\V)=\a_f(\u,\v)+\a_s(\X,\Y)+\c(\llambda,\vcXb-\Y)-\c(\mmu,\ucXb-\X)\\ 
&\BB(\V,q)=(\div\v,q).
\endaligned
\end{equation}
Then problem~\eqref{eq:girato} can be reformulated as follows:
given $\f\in L^2(\Omega)$, $\g\in L^2(\B)$, and $\d\in\Hub$,
find $(\U,p)\in\VV\times\Ldo$ such that
\begin{equation}
\label{eq:misto}
\aligned
&\AA(\U,\V)+\BB(\V,p)=(\f,\v)+(\g,\Y)_{\B}-\c(\mmu,\d)&&\forall\V\in\VV\\
&\BB(\U,q)=0 &&\forall q\in\Ldo.
\endaligned
\end{equation}
In order to verify the well-posedness of~\eqref{eq:misto} we have to
check the two inf-sup conditions, see~\cite[Chapt.~4]{bbf}. 
The inf-sup condition for the bilinear form $\BB$ is the standard inf-sup
condition for the divergence operator: there exists a positive constant
$\beta_{\div}$ such that for all $q\in\Ldo$
\begin{equation}
\label{eq:infsupdiv}
\sup_{\V\in\VV}\frac{\BB(\V,q)}{|||\V|||_{\VV}}=
\sup_{\v\in\Huo}\frac{(\div\v,q)}{\|\v\|_1}\ge\beta_{\div}\|q\|_0.
\end{equation}
The main issue
is to show the invertibility of the operator matrix
\[
\left[
\begin{array}{cc;{2pt/2pt}c}
\Af&\zero&\Cft\\
\zero&\As&-\Cst\\
\hdashline[2pt/2pt]
\Cf&-\Cs&\zero
\end{array}
\right]
\]
\corrig{
on the kernel of $\BB$:
\begin{equation}
\label{eq:kernBB}
\KBB=\{\V\in\VV: \BB(\V,q)=0\ \forall q\in\Ldo\}.
\end{equation}
}
This task can be performed by checking the inf-sup condition for the operator
matrix $\CC=[\Cf\ -\Cs]$ and by showing the
uniform invertibility of the matrix $[\Af\ \zero;\zero\ \As]$ in the kernel of
the operator $\CC$. 
\corrig{
We observe that $\V=(\v,\X,\mmu)\in\KBB$ if $\div\v=0$. Let us set
\begin{equation}
\label{eq:Vdiv0}
\Vo=\{\v\in\Huo: \div\v=0\text{ in }\Omega\}.
\end{equation}
\begin{proposition}
\label{le:infsup}
There exists a constant $\beta_0>0$ such that for all $\mmu\in\LL$ it holds
\begin{equation}
\label{eq:infsup}
\sup_{(\v,\Y)\in\Vo\times\Hub}\frac{\c(\mmu,\vcXb-\Y)}
{(\|\v\|_1^2+\|\Y\|^2_{1,\B})^{1/2}}\ge
\beta_0\|\mmu\|_{\LL}.
\end{equation}
\end{proposition}
}
\begin{proof}
We give the proof for the choices of $\c$ given in~\eqref{eq:defc1}
and~\eqref{eq:defc2}.
In the first case $\c$ is the duality pairing between $\Hub$ and $\Hubd$.
By definition of the norm in the dual space $\Hubd$ we have
\[
\aligned
\|\mmu\|_{\Hubd}&=\sup_{\Y\in\Hub}\frac{\langle\mmu,\Y\rangle}{\|\Y\|_{\Hub}}\\
&\le \sup_{(\v,\Y)\in\corrig{\Vo}\times\Hub}\frac{\langle\mmu,\vcXb-\Y\rangle}
{(\|\v\|_1^2+\|\Y\|_{1,\B}^2)^{1/2}}
\endaligned
\]
which gives~\eqref{eq:infsup}.

The proof is exactly the same for the case of $\c$ given by the scalar product
in $\Hub$.
\end{proof}
The kernel of the operator $\CC$ is given by:
\corrig{
\begin{equation}
\label{eq:kernel}
\K=\{(\v,\Y)\in\Vo\times\Hub:\ \c(\mmu,\vcXb-\Y)=0\ \forall\mmu\in\LL\}.
\end{equation}
}
Then $(\v,\Y)\in\K$ if and only if $\vcXb-\Y=0$ in the sense of $\Hub$.

\begin{proposition}
\label{pr:elker}
There exists $\alpha_0>0$ such that
\begin{equation}
\label{eq:elker}
\a_f(\u,\u)+\a_s(\X,\X)\ge\alpha_0 (\|\u\|^2_1+\|\X\|^2_{1,\B})
\quad \forall(\u,\X)\in\K.
\end{equation}
\end{proposition}
\begin{proof}
By definition we have that $\a_f$ is coercive on $\Huo$, hence there exists a
positive constant $c_1$ such that for all $\u\in\Huo$ it holds
\[
\a_f(\u,\u)\ge c_1\|\u\|^2_1.
\]
It remains to bound $\a_s(\X,\X)$; we have
\[
\a_s(\X,\X)=\beta\|\X\|^2_{0,\B}+\gamma\|\Grad_s\X\|^2_{0,\B}
\ge\min(\beta,\gamma)\|\|\X\|^2_{1,\B}.
\]
If $\beta\ne0$ we obtain~\eqref{eq:elker} by setting
$\alpha_0=\min(c_1,\beta,\gamma)$. 
If $\beta=0$ we still obtain the desired estimate,
since $(\u,\X)\in\K$ and $\X=\ucXb$ so that
$\|\X\|_{0,\B}=\|\ucXb\|_{0,\B}\le\|\u\|_0\le C_\Omega|\u|_1$ where
$C_\Omega$ is the Poincar\'e constant and $|\cdot|_1$ denotes the
$H^1$-seminorm. Hence we have
\[
\a_f(\u,\u)+\a_s(\X,\X)\ge c_1\|\u\|^2_1+\gamma\|\Grad_s\X\|^2_{0,\B}
\ge \frac{c_1}2\|\u\|^2_1+\frac{c_1}{2C_\Omega^2}\|\X\|^2_{0,\B}
+\gamma\|\Grad_s\X\|^2_{0,\B}
\]
and~\eqref{eq:elker} follows by taking
$\alpha_0=\min(\frac{c_1}2,\frac{c_1}{2C_\Omega^2},\gamma)$.

Therefore we can determine a constant $\alpha_0>0$ such that~\eqref{eq:elker}
holds true.
\end{proof}
Putting together the results of the above propositions, we obtain the inf-sup
condition for the bilinear form $\AA$, see~\cite{XZ}.
\corrig{
\begin{proposition}
\label{pr:firstinfsup}
There exists a positive constant $\kappa_0$ such that the
following inf-sup condition holds true:
\begin{equation}
\label{eq:firstinfsup}
\inf_{\U\in\KBB}\sup_{\V\in\KBB}\frac{\AA(\U,\V)}{|||\U|||_{\VV}|||\V|||_{\VV}}
\ge\kappa_0.
\end{equation}
\end{proposition}
}}
\begin{remark}
\label{re:nonlinear}
In the proof of the above proposition, the
assumption~\eqref{eq:PK-lin} has been used in order to have the coerciveness of
the bilinear form $\a_s$. The result extends easily to nonlinear cases whenever
the elastic potential energy satisfies the following bound for some positive
constant $\gamma_0$
\[
E(\X)\ge\gamma_0\|\X\|^2_{1,\B}.
\]
\end{remark}
\subsection{Finite element discretization}
\label{se:FE}
Let us consider a family $\tria$ of regular meshes in $\Omega$ and a
family
$\triaB$ of regular meshes in $\B$. We denote by $h_x$ and $h_s$ the
meshsize of $\tria$ and $\triaB$, respectively.
Let $\Vh\subseteq\Huo$ and $\Qh\subseteq\Ldo$ be finite
element spaces which satisfy the usual discrete ellipticity on the kernel and
the discrete inf-sup conditions for the Stokes problem~\cite{bbf}.
Moreover, we set
\begin{equation}
\label{eq:Sh}
\Sh=\{\Y\in\Hub: \Y|_{T}\in\textbf{P}^1(T)\ \forall T\in\triaB\}.
\end{equation}
The natural choice for $\Lh$, corresponding to the case of $\c$ given
by~\eqref{eq:defc2}, is to take $\Lh=\Sh$. This is actually reasonable also
when $\c$ is defined by~\eqref{eq:defc1}, since in this case the duality
pairing can be represented as a scalar product in $L^2(\B)$, that is:
\begin{equation}
\label{eq:defc1h}
\c(\mmu,\Y)=(\mmu,\Y)\quad\forall\mmu\in\Lh, \Y\in\Sh.
\end{equation}
Of course, several other choices for $\Lh$ might be made; we are not going to
investigate them in this paper.

Then the finite element counterpart of Problem~\ref{pb:DLM_sd} reads.
\begin{problem}
\label{pb:saddlepointh}
Let $\Xb\in W^{1,\infty}(\B)^d$ be invertible with Lipschitz inverse
and $\overline\u\in L^\infty(\Omega)$.
Given $\f\in L^2(\Omega)^d$, $\g\in L^2(\B)^d$, and
$\mathbf{d}\in L^2(\B)^d$,
find $\uh\in\Vh$, $\ph\in\Qh$, $\Xh\in\Sh$, and $\lh\in\Lh$ such
that
\begin{equation}
\label{eq:stath}
\aligned
&\a_f(\uh,\v)-(\div\v,ph)+\c(\lh,\vcXb)=(\f,\v) &&\forall\v\in\Vh\\
&(\div\uh,q)=0&&\forall q\in\Qh\\
&\a_s(\Xh,\Y)-\c(\lh,\Y)=(\g,\Y)_{\B}&&\forall\Y\in\Sh\\
&\c(\mmu,\uhcXb-\Xh)= \c(\mmu,\mathbf{d})&&\forall\mmu\in\Lh.
\endaligned
\end{equation}
\end{problem}
\corrigendum{
In order to show existence, uniqueness and stability of the solution of
Problem~\ref{pb:saddlepointh}, we proceed as previously and
rewrite~\eqref{eq:stath} reordering the unknowns:
given $\f\in L^2(\Omega)$, $\g\in L^2(\B)$, and $\d\in\Hub$, find $\uh\in\Vh$,
$\Xh\in\Sh$, $\lh\in\Lh$ and $\ph\in\Qh$ such that
\begin{equation}
\label{eq:giratoh}
\aligned
&\a_f(\uh,\v)+\c(\lh,\vcXb)-(\div\v,\ph)=(\f,\v)
&&\forall\v\in\Vh\\
&\a_s(\Xh,\Y)-\c(\lh,\Y)=(\g,\Y)_{\B}&&\forall\Y\in\Sh\\
&\c(\mmu,\uhcXb-\Xh)= \c(\mmu,\mathbf{d})
&&\forall\mmu\in\Lh\\
&(\div\uh,q)=0&&\forall q\in\Qh
\endaligned
\end{equation}
Using the same notation as in the previous subsection, we set
\[
\VV_h=V_h\times\Sh\times\Lh,
\]
then the finite element counterpart of~\eqref{eq:misto} reads:
given $\f\in L^2(\Omega)$, $\g\in L^2(\B)$, and $\d\in\Hub$,
find $(\U_h,\lh)\in\VV_h\times\Lh$ such that
\begin{equation}
\label{eq:mistoh}
\aligned
&\AA(\U_h,\V)+\BB(\V,\ph)=(\f,\v)+(\g,\Y)_{\B}+\c(\mmu,\mathbf{d})
&&\forall\V\in\VV_h\\
&\BB(\U_h,q)=0 &&\forall q\in\Qh.
\endaligned
\end{equation}
It is well-known that sufficient conditions for existence and uniqueness of
the solution of~\eqref{eq:mistoh} are the discrete 
versions of the two inf-sup conditions~\eqref{eq:infsupdiv}
and~\eqref{eq:firstinfsup}.
The discrete version of~\eqref{eq:infsupdiv} is the standard discrete inf-sup
condition for the divergence operator: since $\Vh\times\Qh$ is stable for the
Stokes equation, there exists a positive constant
$\overline\beta_{\div}$ such that for all $q_h\in\Qh$
\begin{equation}
\label{eq:infsupdivh}
\sup_{\V_h\in\VV_h}\frac{\BB(\V_h,q_h)}{|||\V_h|||_{\VV}}=
\sup_{\v_h\in\Vh}\frac{(\div\v_h,q_h)}{\|\v_h\|_1}\ge
\overline\beta_{\div}\|q_h\|_0.
\end{equation}
}%
\corrig{%
Let us introduce the discrete kernel of $\BB$
\begin{equation}
\label{eq:kernBBh}
\KBBh=\{\V_h\in\VV_h: \BB(\V_h,q_h)=0\ \forall q_h\in\Qh\}.
\end{equation} 
We have that, $\V_h=(\v_h,\Y_h,\mmu_h)\in\KBBh$
if $\v_h$ belongs to the subspace of $V_h$ of the function with zero discrete
divergence:
\begin{equation}
\label{eq:Vdiv0h}
\Voh=\{\v_h\in V_h: (\div\v_h,q_h)=0\ \forall q_h\in\Qh\}.
\end{equation}
}%
The discrete version of~\eqref{eq:firstinfsup} is a consequence of the
following two propositions.
\corrig{
\begin{proposition}
\label{pr:elkerh}
There exists $\alpha_1>0$ independent of $h_x$ and $h_s$ such that
\begin{equation}
\label{eq:elkerh}
\a_f(\uh,\uh)+\a_s(\Xh,\Xh)\ge \alpha_1(\|\uh\|^2_1+\|\Xh\|^2_{1,\B})
\quad\forall (\uh,\Xh)\in\Voh\times\Sh.
\end{equation}
where the discrete kernel $\K_h$ is given by
\[
\K_h=\left
\{(\v_h,\Y_h)\in\Voh\times\Sh:\ \c(\mmu_h,\vhcXb-\Y_h)=0\ \forall\mmu_h\in\Lh
\right\}.
\]
\end{proposition}
}
\begin{proof}
\corrigendum{
As in the continuous case we have for $\beta>0$
\[
\aligned
\a_f(\uh,\uh)+\a_s(\Xh,\Xh)&
\ge c_1\|\uh\|^2_1+
\beta\|\X_h\|^2_{0,\B}+\gamma\|\Grad_s\X_h\|^2_{0,\B}\\
&\ge c_1\|\uh\|^2_1+\min(\beta,\gamma)\|\|\X_h\|^2_{1,\B},
\endaligned
\]
and 
\begin{equation}
\label{eq:stimaVh}
\a_f(\uh,\uh)+\a_s(\Xh,\Xh)\ge c_1\|\uh\|^2_1+\gamma\|\Grad_s\X_h\|^2_{0,\B}
\end{equation}
if $\beta=0$. In this case we need to estimate $\|\Xh\|_{0,\B}$ by means
of the other terms appearing in the right hand side of the last two
inequalities.}
This part of the proof depends on the definition of $\c$. 

\emph{First case.}
Let us assume first that $\c$ is given by~\eqref{eq:defc1}. 
Taking into account~\eqref{eq:defc1h}, we have that 
\corrigendum{$(\uh,\Xh)\in\K_h$} is
characterized by $\Xh=P_0(\uhcXb)$ where $P_0$ represents the $L^2$
projections onto $S_h$.
Hence we obtain
\[
\aligned
\|\Xh\|_{0,\B}&=\|P_0(\uhcXb)\|_{0,\B}\\
&\le\|\uhcXb\|_{0,\B}+\|\uhcXb-P_0(\uhcX)\|_{0,\B}\\
&\le\|\uhcXb\|_{0,\B}+Ch_s|\uhcXb|_{1,\B}
\le\|\uh\|_{0,\Omega}+Ch_s|\uh|_{1,\Omega}.
\endaligned
\]
Therefore we can determine a constant $C>0$ such that
\corrigendum{
\[
\a_f(\uh,\uh)+\a_s(\Xh,\Xh)
\ge C \left(\|\uh\|^2_1+\|\Xh\|^2_{1,\B}\right).
\]}
\emph{Second case.}
In the case of $\c$ given by~\eqref{eq:defc2}, the fact that 
\corrigendum{$(\uh,\Xh)\in\K_h$}
implies that $\Xh=P_1(\uhcXb)$ where $P_1$ stands for the $H^1$ projections
onto $\Sh$, so that we obtain
\[
\aligned
\|\Xh\|_{0,\B}&=\|P_1(\uhcXb)\|_{0,\B}\le\|P_1(\uhcXb)\|_{1,\B}\\
&\le\|\uhcXb\|_{1,\B}\le\|\uh\|_1.
\endaligned
\]
These inequalities together with~\eqref{eq:stimaVh} give the
desired estimate~\eqref{eq:elkerh}.
\end{proof}

If $\c$ is given by~\eqref{eq:defc1} (which has the discrete
counterpart~\eqref{eq:defc1h}), we make the additional assumption 
that the mesh sequence $\triaB$ is such that the $L^2$-projection $P_0$ 
from $\Hub$ onto $\Sh$ is $H^1$-stable, that is 
\begin{equation}
\label{eq:H1stab}
|P_0\v|_{1,\B}\le c|\v|_{1,\B} \quad\forall\v\in\Hub,
\end{equation}
where $|\cdot|_{1,\B}$ is the $H^1$-seminorm.

\begin{proposition}
\label{pr:infsuph}
There exists a constant $\beta_1>0$ independent of $h_x$ and $h_s$
such that for all $\mmu_h\in\Lh$ it holds
true
\corrig{
\begin{equation}
\label{eq:infsuph}
\sup_{(\v_h,\Y_h)\in\corrig{\Voh}\times\Sh}\frac {\c(\mmu_h,\vhcXb-\Y_h)}
{(\|\v_h\|^1_1+\|\Y_h\|^2_{1,\B})^{1/2}}\ge
\beta_1\|\mmu_h\|_{\LL}.
\end{equation}
}
\end{proposition}

\begin{proof}
\emph{First case.} Let $\c$ be given by~\eqref{eq:defc1}, so
that~\eqref{eq:defc1h} holds true.
Then we have to show that there exists $\beta_1>0$ independent of $h$ such that
\[
\sup_{(\v_h,\Y_h)\in\Voh\times\Sh}
\frac{(\mmu_h,\vhcXb-\Y_h)}{\|\V_h\|_{\VV}}\ge\beta_1\|\mmu_h\|_{\LL}.
\]
By definition of the norm in the dual space $\LL=\Hubd$, there exists
$\widetilde\Y\in\Hub$ such that
\begin{equation}
\label{eq:inizio1}
\|\mmu_h\|_{\LL}=\sup_{\Y\in\Hub}\frac{(\mmu_h,\Y)}{\|\Y\|_{1,\B}}
=\frac{(\mmu_h,\widetilde\Y)}{\|\widetilde\Y\|_{1,\B}}
=\frac{(\mmu_h,P_0\widetilde\Y)}{\|\widetilde\Y\|_{1,\B}}.
\end{equation}
where $P_0$ denotes the projection operator from $\Hub$ into
$\Lh=\Sh$.

Well-known properties of $P_0$ and~\eqref{eq:H1stab} imply
\[
\aligned
&\|P_0\widetilde\Y\|_{0,\B}\le\|\widetilde\Y\|_{0,\B}\\
&|P_0\widetilde\Y|_{1,\B}\le c|\widetilde\Y|_{1,\B}\\
&\|\widetilde\Y-P_0\widetilde\Y\|_{0,\B}\le Ch_s|\widetilde\Y|_{1,\B}.
\endaligned
\]
Therefore there exists a constant $C$ such that
$\|P_0\widetilde\Y\|_{1,\B}\le C\|\widetilde\Y\|_{1,\B}$.
This last inequality inserted in~\eqref{eq:inizio1} gives
\corrigendum{
\[
\aligned
\|\mmu_h\|_{\LL}&\le
C\frac{(\mmu_h,P_0\widetilde\Y)}{\|P_0\widetilde\Y\|_{1,\B}}\\
&\le C\sup_{\Y_h\in\Sh}
\frac{(\mmu_h,\Y_h)}{\|\Y_h\|_{1,\B}}
\le C\sup_{(\v_h,\Y_h)\in\corrig{\Voh}\times\Sh} 
\frac{\c(\mmu_h,\vhcXb-\Y_h)}{(\|\v_h\|^1_1+\|\Y_h\|^2_{1,\B})^{1/2}}.
\endaligned
\]
}
\emph{Second case.}
Let us now consider $\c$ given by~\eqref{eq:defc2}, hence it is the scalar
product in $\Hub$.
By definition of the norm in $\Hub$ and of the $H^1$-projection operator
$P_1$ we have:
\corrigendum{
\[
\aligned
\|\mmu_h\|_{\LL}&=\sup_{\Y\in\Hub}\frac{\c(\mmu,\Y)}{\|\Y\|_{1,\B}}
\le\sup_{\Y\in\Hub}\frac{\c(\mmu,P_1\Y)}{\|P_1\Y\|_{1,\B}}\\
&\le \sup_{\Y_h\in\Sh}\frac{\c(\mmu,\Y_h)}{\|\Y_h\|_{1,\B}}
\le \sup_{(\v_h,\Y_h)\in\corrig{\Voh}\times\Sh} \frac{\c(\mmu_h,\vhcXb-\Y_h)}
{(\|\v_h\|^1_1+\|\Y_h\|^2_{1,\B})^{1/2}}.
\endaligned
\]
}
\end{proof}
\corrigendum{
Putting together the results of the above propositions, we obtain the 
discrete inf-sup condition for $\AA$, see~\cite{XZ}.
\begin{proposition}
\label{pr:firstinfsuph}
There exists a positive constant $\kappa_1$ such that the
following inf-sup condition holds true:
\begin{equation}
\label{eq:firstinfsuph}
\inf_{\U_h\in\corrig{\KBBh}}\sup_{\V_h\in\corrig{\KBBh}}
\frac{\AA(\U_h,\V_h)}{|||\U_h|||_{\VV}|||\V_h|||_{\VV}}
\ge\kappa_1.
\end{equation}
\end{proposition}
}
\begin{remark}
Condition~\eqref{eq:H1stab} has been widely studied in the literature.
It can be easily obtained by using an inverse inequality on quasi-uniform
meshes.
Weaker assumptions than the quasi-uniformity of the mesh have been
investigated in several papers,
see for example~\cite{CrTh1987,BPS2001,CC2002,CC2004,BankYser2014}.
In particular, the
stability of the $L^2$-projection in $H^1$ has been proved in~\cite{CrTh1987}
under the assumption that neighboring element-sizes obey a global
growth-condition and in~\cite{BPS2001} in the case of locally
quasiuniform meshes. These conditions have been weakened in~\cite{CC2002},
while~\cite{CC2004} extends the result to meshes generated by red-green-blue
refinements in adaptive procedures for piecewise linear finite elements.
Recently~\cite{BankYser2014} has improved the previous results to cover the
case of many commonly used adaptive meshing strategies. More general mesh
refinements are considered in~\cite{gaspacio}.
\end{remark}

From the theory of the discretization of saddle point problems,
the above propositions yield the following error estimate theorem
(see~\cite[Th.~5.2.1]{bbf}).
\begin{thm}
\label{th:errthick}
Let $(\u,p,\X,\llambda)$ and $(\uh,\ph,\Xh,\lh)$ be solutions of
Problems~\ref{pb:saddlepoint} and~\ref{pb:saddlepointh}. Under the
assumption~\eqref{eq:H1stab} if $\c$ is given by~\eqref{eq:defc1}, the
following optimal error estimate holds true:
\[
\aligned
\|\u-\uh\|_1&+\|p-\ph\|_0+\|\X-\Xh\|_{1,\B}+\|\llambda-\lh\|_{\LL}\\
\le& C\inf_{\substack{\v\in\Vh\\ q\in\Qh\\ \Y\in\Sh\\ \mmu\in\Sh}}
\left(\|\u-\v\|_1+\|p-q\|_0+\|\X-\Y\|_{1,\B}+\|\llambda-\mmu\|_{\LL}\right).
\endaligned
\] 
\end{thm}

\section{The case of a thin solid immersed in a fluid}
\label{se:thin}
In this section we consider the case of thin structures with very small
constant thickness $t_s$, so that we assume that the physical quantities
depend only on variables along the middle section of the structure and are
constant in the normal direction.
In order to maintain the same notation in the final
formulation of the problem, the region occupied by the solid is
$\Os_t\times]-t_s,t_s[$, where $\Os_t$ is a subset of $\Omega$ of codimension
one (a surface in the 3D case or a curve in the 2D one). Therefore we have
that the reference domain $\B$ is a Lipschitz subdomain of $\RE^{d-1}$ and the
deformation gradient $\F:\B\to\RE^{d\times (d-1)}$ is such that
\[
|\F|=\left|\frac{\partial\X}{\partial s}\right| \text{ if }d=2,\qquad
|\F|=\left|
\frac{\partial\X}{\partial s_1}\times\frac{\partial\X}{\partial s_2}
\right|\text{ if }d=3,
\]
$s$, $s_1$ and $s_2$ being the parametric variables in $\B$.

Following the same arguments as in~\cite{BCG2011},
Equations~\eqref{eq:virtualwork2}-\eqref{eq:divfree} can be written
in the following form:
\begin{equation}
\aligned
  &\rho_f\frac d {dt}(\u(t),\v)+b(\u(t),\u(t),\v)+a(\u(t),\v)\qquad&&\\
  &\qquad-(\div\v,p(t))=
  \langle\mathbf{F}_1(t),\v\rangle+\langle\mathbf{F}_2(t),\v\rangle
   &&\quad\forall\v\in\Huo
\\
  &(\div\u(t),q)=0&&\quad\forall q\in\Ldo
\\
  &\langle\mathbf{F}_1(t),\v\rangle=
   -\dr\int_\B\frac{\partial^2\X}{\partial t^2}\v(\X(\s,t))\,\ds
      &&\quad\forall\v\in\Huo
\\
  &\langle\mathbf{F}_2(t),\v\rangle= - \int_\B \P(\F(\s,t)): \nabla_s
\v(\X(\s,t))\,\ds
      &&\quad\forall\v\in\Huo.
\endaligned
\end{equation}
Here $\u$ and $p$ represent velocity and pressure of the fluid, respectively,
$\dr=(\rho_s-\rho_f)t_s$, and $\P=t_s\tilde\P$ where $\tilde\P$ is obtained
from~\eqref{eq:Piola} with the necessary modifications to cover the present
situation.
Moreover, the motion of the thin structure is governed by the following condition 
\[
\u(\x,t)=\frac{\partial\X(\s,t)}{\partial t}\Big|_{\x=\X(\s,t)}.
\]
Then the problem has the same form as Problem~\ref{pb:pbvar} 
by rewriting the above body motion constraint variationally as
\[
\c\left(\mmu,\ucX-\frac{\partial\X}{\partial t}(t)\right)=0
\]
for all $\mmu$ in a suitably defined functional space $\LL$.
Assuming that $\X(t)\in W^{1,\infty}(\B)$ is one to one, $\ucX$
represents the trace of $\u$ along $\Os_t$.
Therefore $\ucX\in\HhalfB$; we set $\LL=(\HhalfB)'$ and
$\c:\LL\times\HhalfB\to\RE$ given by
\begin{equation}
\label{eq:defcthin}
\c(\mmu,\z)=\langle\mmu,\z\rangle\quad\forall\mmu\in\LL,\z\in\HhalfB,
\end{equation}
where $\langle\cdot,\cdot\rangle$ is the duality pairing between $\HhalfB$
and $\LL=(\HhalfB)'$.

With this definition, we can perform the same stability analysis as in
Sections~\ref{se:pb} and~\ref{se:dt}
and show that Propositions~\ref{pr:stab-cont}
and~\ref{pr:energy_sd} hold true also in this case. 
In the following we analyze the well-posedness of Problem~\ref{pb:saddlepoint}
and its finite element discretization. The discussion will be carried on using
the same arguments as in Sections~\ref{se:continuous} and~\ref{se:FE} relying
on the formulation~\eqref{eq:girato} written in~\eqref{eq:misto} as a saddle
point problem. 
\corrigendum{
As before the inf-sup condition for the bilinear form $\BB$ is the standard
inf-sup condition for the divergence operator, therefore~\eqref{eq:infsupdiv}
and~\eqref{eq:infsupdivh} hold also in the present case. 
It remains to show the analogous of Propositions~\ref{le:infsup}
and~\ref{pr:elker} and their discrete counterparts in the case of
thin structures. As in the previous sections, these two propositions imply the 
required inf-sup condition for $\AA$. 
}%
The following inf-sup conditions ensure existence and uniqueness of 
the solution~\eqref{eq:misto} in the case of thin structures.
\corrigendum{
\begin{proposition}
\label{pr:elker2}
Let $\c$ and $\K$ be given by~\eqref{eq:defcthin} and~\eqref{eq:kernel},
respectively. Then there exists $\alpha_0>0$ such that
\begin{equation}
\label{eq:elker2}
\a_f(\u,\u)+\a_s(\X,\X)\ge\alpha_0(\|\u\|^2_1+\|\X\|^2_{1,\B})\quad
\forall(\u,\X)\in\K.
\end{equation}
\end{proposition}
}
\begin{proof}
The proof is the same except for the case $\beta=0$. The fact that
\corrigendum{$(\u,\X)\in\K$} implies again that $\X=\ucXb$, that is
$\X$ is the trace of $\u$ on $\Gb=\Xb(\B)$ and the bound
$\|\X\|_{0,\B}\le C\|\u\|_1$ is a consequence of the
trace theorem in $\Huo$.
\end{proof}
\corrig{
The following auxiliary lemma will be used to show the inf-sup condition
for $\c$.
\begin{lemma}
\label{le:tracce}
For all $\z\in\HhalfB$ there exists $\w\in\Vo$ such that $\w(\Xb)=\z$
with $\|w\|_1\le c\|\z\|_{\HhalfB}$.
\end{lemma}
\begin{proof}
For simplicity, we sketch the proof in the case of a closed orientable
manifold $\B$. The general case can be reduced to this one by introducing
suitable cuts of $\Omega$.
Thanks to the surjectivity of the trace operator from $\Huo$ to
$(H^{1/2}(\overline\Gamma))^d$, there exists $\w_0\in\Huo$ such that
$\w_0(\Xb)=\z$ with $\|\w_0\|_1\le c\|\z\|_{\HhalfB}$. Since $\B$ is a closed
manifold of codimension 1, $\Omega$ is divided into two subsets $\Omega_1$ and
$\Omega_2$ by $\B$.
Due to the inf-sup condition~\eqref{eq:infsupdiv}, there exist 
$\w_i\in H^1_0(\Omega_i)^d$ for $i=1,2$ such that
$\div\w_i=\div\w_0$ in $\Omega_i$ with
$\|\w_i\|_{H^1(\Omega_i)}\le c_i\|\w_0\|_1$. Setting
\[
\w=\left\{
\begin{array}{ll}
\w_0-\w_1 &\quad\text{in }\Omega_1\\
\w_0-\w_2 &\quad\text{in }\Omega_2,
\end{array}
\right.
\]
we have that $\w\in\Vo$ and $\|\w\|_1\le c\|\z\|_{\HhalfB}$ as required.
\end{proof}
}
\begin{proposition}
\label{pr:infsup2}
Let $\c$ be given by~\eqref{eq:defcthin}, then
there exists a constant $\beta_0>0$ such that for all $\mmu\in\LL$ it holds
true
\corrigendum{
\begin{equation}
\label{eq:infsup2}
\sup_{(\v,\Y)\in\corrig{\Vo}\times\Hub}\frac {\c(\mmu,\vcXb-\Y)}
{(\|\v\|^2_1+\|\Y\|^2_{1,\B})^{1/2}}\ge \beta_0\|\mmu\|_{\LL}.
\end{equation}
}
\end{proposition}
\begin{proof}
Since $\LL=(\HhalfB)'$, we have the following definition of the norm in $\LL$:
\[
\|\mmu\|_{\LL}=\sup_{\z\in\HhalfB} \frac{\langle\mmu,\z\rangle}{\|\z\|_{\HhalfB}}
=\sup_{\z\in\HhalfB} \frac{\c(\mmu,\z)}{\|\z\|_{\HhalfB}}.
\]
Let us consider a maximizing sequence $\{\z_n\}_{n\in\NA}$ such that 
\[
\lim_{n\to\infty}
\frac{\c(\mmu,\z_n)}{\|\z_n\|_{\HhalfB}}=\|\mmu\|_{\LL}.
\]
\corrig{
Thanks to Lemma~\ref{le:tracce}, there exists $\u_n\in\Vo$ such that
$\u_n(\Xb)=\z_n$ with $\|\u_n\|_1\le c\|z_n\|_{\HhalfB}$.}
Hence we obtain the desired inequality~\eqref{eq:infsup2} as follows
\corrigendum{
\[
\aligned
\sup_{(\v,\Y)\in\corrig{\Vo}\times\Hub}&\frac{\c(\mmu,\vcXb-\Y)}{\|\V\|_{\VV}}
\ge \sup_{\v\in\corrig{\Vo}}\frac{\c(\mmu,\vcXb)}{\|\v\|_1}\\
&\ge \frac{\c(\mmu,\u_n(\Xb))}{\|\u_n\|_1}
\ge\frac1c \frac{\c(\mmu,\z_n)}{\|\z_n\|_{\HhalfB}}
\ge\frac1{2c}\|\mmu\|_{\LL}.
\endaligned
\]
}
\end{proof}
Let us now introduce a finite element discretization of Problem~\ref{pb:pbvar}
with $\c$ given by~\eqref{eq:defcthin}.
With the same notation as in Section~\ref{se:FE}, we set $\LL_h=\Sh\subset\LL$.
Then we have again that the duality pairing of regular elements in $\LL$ can be
computed as the scalar product in $L^2(\B)$.
Let us show the discrete inf-sup conditions which ensure existence and
uniqueness of the discrete solution together with optimal error estimate.
\corrigendum{
\begin{proposition}
\label{pr:elkerh2}
There exists $\alpha_1>0$ independent of $h_x$ and $h_s$ such that
\begin{equation}
\label{eq:elkerh2}
\a_f(\uh,\uh)+\a_s(\Xh,\Xh)\ge\alpha_0(\|\uh\|^2_1+\|\Xh\|^2_{1,\B})\quad
\forall(\uh,\Xh)\in\K_h.
\ge\alpha_1,
\end{equation}
where the discrete kernel $\K_h$ is given by
\[
\K_h=\left
\{(\v_h,\Y_h)\in\corrig{\Voh}\times\Sh:\ \c(\mmu_h,\vhcXb-\Y_h)=0\
\forall\mmu_h\in\Lh \right\}.
\]
\end{proposition}
}
\begin{proof}
\corrigendum{
The definitions of $\a_f$ and $\a_s$ yield
\[
\a_f(\uh,\uh)+\a_s(\Xh,\Xh)\ge c_1\|\uh\|^2_1+\beta\|\Xh\|^2_{0,\B}+
\gamma\|\Grad_s\Xh\|^2_{0,\B}
\]
}
and in the case $\beta=0$ we need to bound $\|\Xh\|^2_{0,\B}$.
Since \corrigendum{$(\uh,\Xh)\in\K_h$}, we have by definition that
$\c(\mmu,\uhcXb-\Xh)=(\mmu,\uhcXb-\Xh)_\B=0$ for all $\mmu\in\Sh$. 
Therefore $\Xh=P_0(\uhcXb)$ where $P_0$ is the $L^2$-projection onto $\Sh$, so
that
\[
\aligned
\|\Xh\|_{0,\B}&=\|P_0(\uhcXb)\|_{0,\B}
\le\|\uhcXb\|_{0,\B}+\|P_0(\uhcXb)-\uhcXb\|_{0,\B}\\
&\le\|\uhcXb\|_{0,\B}+Ch_s^{1/2}\|\uhcXb\|_{1/2,\B}
\le\|\uhcXb\|_{0,\B}+Ch_s^{1/2}\|\uh\|_{1/2,\Gb}\\
&\le C\|\uh\|_1.
\endaligned
\]
\corrigendum{This concludes the proof in the case $\beta=0$ as well.}
\end{proof}
\corrig{
The following approximation property of elements of $\Vo$ with functions
in $\Voh$ will be useful for the proof of the discrete inf-sup condition for
$\c$.
\begin{lemma}
\label{le:approxKKB}
Let us assume that the domain $\Omega$ is convex. Then for all $\bar\u\in\Vo$
there exists $\bar\u_h\in\Voh$ such that
\[
\aligned
&\|\bar\u-\bar\u_h\|_0\le Ch_x\|\bar\u\|_1\\
&\|\bar\u_h\|_1\le C\|\bar\u\|_1.
\endaligned
\]
\end{lemma}
\begin{proof}
Given $\bar\u\in\Vo$, the pair $(\bar\u,\bar p)\in\Huo\times\Ldo$ 
is the solution of the following Stokes problem
\[
\aligned
&(\nabla\bar\u,\nabla\v)-(\div\v,\bar p)=(\nabla\bar\u,\nabla\v)&&
\forall\v\in\Huo\\
&(\div\u,q)=0&&\forall q\in\Ldo,
\endaligned
\]
with $\bar p=0$. Let $(\bar\u_h,\bar p_h)\in\Vh\times\Qh$ be the solution
of the associated discrete problem
\[
\aligned
&(\nabla\bar\u_h,\nabla\v_h)-(\div\v_h,\bar p_h)=(\nabla\bar\u,\nabla\v_h)&&
\forall\v_h\in\Vh\\
&(\div\u_h,q_h)=0&&\forall q_h\in\Qh.
\endaligned
\]
We have $\|\bar\u_h\|_1+\|\bar p_h\|_0\le C\|\bar\u\|_1$.
To estimate the $L^2$-norm of the difference $\bar\u-\bar\u_h$, we introduce
the dual problem
\begin{equation}
\label{eq:dual}
\aligned
&(\nabla\w,\nabla\v)-(\div\v,r)=(\bar\u-\bar\u_h,\v)&&\forall\v\in\Huo\\
&(\div\w,q)=0&&\forall q\in\Ldo.
\endaligned
\end{equation}
Since $\Omega$ is convex, the following a priori estimate for the
solution of~\eqref{eq:dual} holds true
\[
\|\w\|_2+\|r\|_1\le C\|\bar\u-\bar\u_h\|_0.
\]
Let $\w^I\in\Vh$ and $r^I\in\Qh$ interpolate $\w$ and $r$ respectively, with
\[
\|\w-\w_I\|_1+\|r-r^I\|_0\le Ch(\|\w\|_2+\|r\|_1),
\]
then standard computations yield
\[
\aligned
\|\bar\u-\bar\u_h\|_0^2&=(\nabla\w,\nabla(\bar\u-\bar\u_h))-
(\div(\bar\u-\bar\u_h),r)\\
&=(\nabla(\w-\w^I),\nabla(\bar\u-\bar\u_h))+(\div(\w-\w^I),\bar p-\bar p_h)
-(\div(\bar\u-\bar\u_h),r-r^I)\\
&\le\|\w-\w^I\|_1\left(\|\bar\u-\bar\u_h\|_1+\|\bar p-\bar p_h\|_0\right)+
\|\bar\u-\bar\u_h\|_1\|r-r^I\|\\
&\le Ch\|\bar\u-\bar\u_h\|_0(\|\bar\u-\bar\u_h\|_1+\|\bar p-\bar p_h\|_0)\\
&\le Ch\|\bar\u-\bar\u_h\|_0\|\bar\u\|_1.
\endaligned
\] 
\end{proof}
}
\corrigendum{
\begin{proposition}
\label{pr:infsuph2}
\corrig{Let us assume that the domain $\Omega$ is convex.}
If $h_x/h_s$ is sufficiently small and the mesh $\triaB$ is quasi-uniform,
then there exists a constant $\beta_1>0$ independent of $h_x$ and $h_s$
such that for all $\mmu_h\in\Lambda_h$ it holds true
\begin{equation}
\label{eq:infsuph2}
\sup_{(\v_h,\Y_h)\in\corrig{\Voh}\times\Sh}\frac{\c(\mmu_h,\vhcXb-\Y_h)}
{(\|\v_h\|^2_1+\|\Y_h\|^2_{1,\B})^{1/2}}\ge
\beta_1\|\mmu_h\|_{\LL}.
\end{equation}
\end{proposition}
}
\begin{proof}
In the proof of Proposition~\ref{pr:infsup2} we have shown that
\[
\sup_{\v\in\corrig{\Vo}}\frac{\c(\mmu,\vcXb)}{\|\v\|_1}\ge
\frac1{2c}\|\mmu\|_{\LL}
\]
for all $\mmu\in\LL$.
Let us fix $\mmu_h\in\Lh$. Since $\Lh\subset\LL$, the above inequality holds
true also for $\mmu_h$. Let $\bar\u\in\corrig{\Vo}$ be the element in
$\corrig{\Vo}$ where the above supremum is attained, hence 
\[
\c(\mmu_h,\bar\u(\Xb))\ge \frac1{2c}\|\mmu\|_{\LL}\|\bar\u\|_1.
\]
\corrig{
Let $\bar\u_h\in\Voh$ be the approximation of $\bar\u$ introduced in
Lemma~\ref{le:approxKKB},
then, we write
}
\[
\c(\mmu_h,\bar\u_h(\Xb))=
\c(\mmu_h,\bar\u(\Xb))+\c(\mmu_h,\bar\u_h(\Xb)-\bar\u(\Xb)).
\]
We bound the second term on the right hand side, using a trace
theorem~\cite[Th.1.6.6]{BrennerScott} as follows
\[
\|\bar\u_h(\Xb)-\bar\u(\Xb)\|_{0,\B}\le 
C(\|\bar\u_h-\bar\u\|_{0,\Omega}\|\bar\u_h-\bar\u\|_{1,\Omega})^{1/2}\le
C h_x^{1/2} \|\bar\u\|_{1}
\]
and the first one by means of the following inverse inequality
\[
\|\mmu_h\|_{0,\B}\le Ch_s^{-1/2}\|\mmu_h\|_{\LL}.
\]
Hence we obtain
\[
\aligned
\c(\mmu_h,\bar\u_h(\Xb))&=
\c(\mmu_h,\bar\u(\Xb))+\c(\mmu_h,\bar\u_h(\Xb)-\bar\u(\Xb))\\
&\ge\frac1{2c}\|\mmu\|_{\LL}\|\bar\u\|_1
-C\|\mmu_h\|_{0,\B}h_x^{1/2}\|\bar\u\|_1\\
&\ge\|\mmu\|_{\LL}\|\bar\u\|_1
\left(\frac1{2c}-C\left(\frac{h_x}{h_s}\right)^{1/2}\right).
\endaligned
\]
The desired inequality follows easily from
\[
\sup_{\v\in\Vh}\frac{\c(\mmu_h,\vcXb)}{\|\v\|_1}
\ge \frac{\c(\mmu,\bar\u_h(\Xb))}{\|\bar\u_h\|_1}\ge
\left(\frac1{2c}-C\left(\frac{h_x}{h_s}\right)^{1/2}\right)
\frac{\|\bar\u\|_1}{\|\bar\u_h\|_1}\|\mmu_h\|_{\LL}
\]
if $(h_x/h_s)^{1/2}\le1/(2cC)$.

\end{proof}
\section*{Acknowledgements}
The authors were partially supported by PRIN/MIUR,
GNCS/INDAM, and IMATI/CNR.

\let\v\oldv 
\bibliographystyle{plain}
\bibliography{ref}

\begin{thebibliography}{10}

\bibitem{BankYser2014}
R.~E. Bank and H.~Yserentant.
\newblock On the {$H^1$}-stability of the {$L_2$}-projection onto finite
  element spaces.
\newblock {\em Numer. Math.}, 126(2):361--381, 2014.

\bibitem{bbf}
D.~Boffi, F.~Brezzi, and M.~Fortin.
\newblock {\em Mixed Finite Element Methods and Applications}, volume~44 of
  {\em Springer Series in Computational Mathematics}.
\newblock Springer-Verlag, New York, 2013.

\bibitem{BCG2011}
D.~Boffi, N.~Cavallini, and L.~Gastaldi.
\newblock Finite element approach to immersed boundary method with different
  fluid and solid densities.
\newblock {\em Math. Models Methods Appl. Sci.}, 21(12):2523--2550, 2011.

\bibitem{bcg2015}
D.~Boffi, N.~Cavallini, and L.~Gastaldi.
\newblock The finite element immersed boundary method with distributed
  {L}agrange multiplier.
\newblock SIAM J. Numer. Anal., to appear. arXiv:1407.5261 [math.NA], 2015.

\bibitem{feibm1}
D.~Boffi and L.~Gastaldi.
\newblock A finite element approach for the immersed boundary method.
\newblock {\em Comput. \& Structures}, 81(8-11):491--501, 2003.
\newblock In honour of Klaus-J{\"u}rgen Bathe.

\bibitem{BGHM3AS}
D.~Boffi, L.~Gastaldi, and L.~Heltai.
\newblock Numerical stability of the finite element immersed boundary method.
\newblock {\em Math. Models Methods Appl. Sci.}, 17(10):1479--1505, 2007.

\bibitem{BGHP}
D.~Boffi, L.~Gastaldi, L.~Heltai, and C.~S. Peskin.
\newblock On the hyper-elastic formulation of the immersed boundary method.
\newblock {\em Comput. Methods Appl. Mech. Engrg.}, 197(25-28):2210--2231,
  2008.

\bibitem{BPS2001}
J.~H. Bramble, J.~E. Pasciak, and O.~Steinbach.
\newblock On the stability of the {$L^2$} projection in {$H^1(\Omega)$}.
\newblock {\em Math. Comp.}, 71(237):147--156 (electronic), 2002.

\bibitem{BrennerScott}
S.~C. Brenner and R.~L. Scott.
\newblock {\em The mathematical theory of finite element methods}, volume~15 of
  {\em Texts in Applied Mathematics}.
\newblock Springer, New York, third edition, 2008.

\bibitem{CC2002}
C.~Carstensen.
\newblock Merging the {B}ramble-{P}asciak-{S}teinbach and the
  {C}rouzeix-{T}hom\'ee criterion for {$H^1$}-stability of the
  {$L^2$}-projection onto finite element spaces.
\newblock {\em Math. Comp.}, 71(237):157--163 (electronic), 2002.

\bibitem{CC2004}
C.~Carstensen.
\newblock An adaptive mesh-refining algorithm allowing for an {$H^1$} stable
  {$L^2$} projection onto {C}ourant finite element spaces.
\newblock {\em Constr. Approx.}, 20(4):549--564, 2004.

\bibitem{CrTh1987}
M.~Crouzeix and V.~Thom{\'e}e.
\newblock Resolvent estimates in {$l_p$} for discrete {L}aplacians on irregular
  meshes and maximum-norm stability of parabolic finite difference schemes.
\newblock {\em Comput. Methods Appl. Math.}, 1(1):3--17, 2001.

\bibitem{gaspacio}
F.D. Gaspoz, C.-J. Heine, and K.G. Siebert.
\newblock Optimal grading of the newest vertex bisection and {H}1-stability of
  the {L}2-projection.
\newblock {\em IMA J. of Numerical Analysis}, 2015.
\newblock To appear. doi:10.1093/imanum/drv044.

\bibitem{girglo1995}
V.~Girault and R.~Glowinski.
\newblock Error analysis of a fictitious domain method applied to a {D}irichlet
  problem.
\newblock {\em Japan J. Indust. Appl. Math.}, 12(3):487--514, 1995.

\bibitem{girglopan}
V.~Girault, R.~Glowinski, and T.-W. Pan.
\newblock A fictitious-domain method with distributed multiplier for the
  {S}tokes problem.
\newblock In {\em Applied nonlinear analysis}, pages 159--174. Kluwer/Plenum,
  New York, 1999.

\bibitem{glokuz2007}
R.~Glowinski and Yu. Kuznetsov.
\newblock Distributed {L}agrange multipliers based on fictitious domain method
  for second order elliptic problems.
\newblock {\em Comput. Methods Appl. Mech. Engrg.}, 196(8):1498--1506, 2007.

\bibitem{heltai}
L.~Heltai.
\newblock On the stability of the finite element immersed boundary method.
\newblock {\em Comput. \& Structures}, 86(7-8):598--617, 2008.

\bibitem{costanzo}
L.~Heltai and F.~Costanzo.
\newblock Variational implementation of immersed finite element methods.
\newblock {\em Comput. Methods Appl. Mech. Engrg.}, 229/232:110--127, 2012.

\bibitem{kamensky2}
D.~Kamensky, J.A. Evans, and M.-C. Hsu.
\newblock Stability and conservation properties of collocated constraints in
  immersogeometric fluid{-}thin structure interaction analysis.
\newblock Comm.\ in Comput.\ Physics, to appear, 2015.

\bibitem{kamensky}
D.~Kamensky, M.-C. Hsu, D.~Schillinger, J.A. Evans, A.~Aggarwal, Y.~Bazilevs,
  M.S. Sacks, and T.J.R. Hughes.
\newblock An immersogeometric variational framework for fluid-structure
  interaction: application to bioprosthetic heart valves.
\newblock {\em Comput. Methods Appl. Mech. Engrg.}, 284:1005--1053, 2015.

\bibitem{PeAN}
C.~S. Peskin.
\newblock The immersed boundary method.
\newblock {\em Acta Numer.}, 11:479--517, 2002.

\bibitem{WangLiu}
X.~Wang and W.K. Liu.
\newblock Extended immersed boundary method using {FEM} and {RKPM}.
\newblock {\em Comput. Methods Appl. Mech. Engrg.}, 193:1305--1321, 2004.

\bibitem{XZ}
Jinchao Xu and Ludmil Zikatanov.
\newblock Some observations on {B}abu\v ska and {B}rezzi theories.
\newblock {\em Numer. Math.}, 94(1):195--202, 2003.

\end{thebibliography}

\end{document}